\documentclass{article}
\usepackage[german,english]{babel} 
\usepackage[utf8]{inputenc}
\usepackage{mathtools}
\usepackage{amssymb}
\usepackage{amsthm}

\newtheorem{theorem}{Theorem}[section]
\newtheorem{definition}[theorem]{Definition}
\newtheorem{corollary}[theorem]{Corollary}
\newtheorem{lemma}[theorem]{Lemma}
\newtheorem{proposition}[theorem]{Proposition}

\theoremstyle{remark}
\newtheorem*{remark}{Remark}

\newcommand{\naturals}{\ensuremath{\mathbb{N}}}

\title{A total Solovay reducibility and totalizing of the notion of speedability}
\author{Wolfgang Merkle and Ivan Titov}
\date{March 2021}

\begin{document}

\maketitle

\newcommand{\defhigh}[1]{\textsc{#1}}

\begin{abstract}
While the set of Martin-Löf random left-c.e.\ reals is equal to the maximum degree of Solovay reducibility, Miyabe, Nies and Stephan~\cite{miyabe-etal-2018} have shown that the left-c.e.\ Schnorr random reals are not closed upwards under Solovay reducibility. Recall that for two left-c.e.\ reals~$\alpha$ and~$\beta$, the former is Solovay reducible to the latter in case there is a partially computable function~$\varphi$ and constant~$c$ such that for all rational numbers~$q < \beta$ we have
\[\alpha - \varphi(q) < c(\beta - q).\]
By requiring the translation function $\varphi$ to be total, we introduce a total version of Solovay reducibility that implies Schnorr reducibility. Accordingly, by Downey and Griffiths~\cite{downey-griffiths-2004}, the set of Schnorr random left-c.e.\ reals is closed upwards relative to total Solovay reducibility.

Furthermore, we observe that the notion of speedability introduced by Merkle and Titov~\cite{Merkle-Titov-2020} can be equivalently characterized via partial computable translation functions in a way that resembles Solovay reducibility. 
By requiring the translation function to be total, we obtain the concept of total speedability. Like for speedability, this notion does not depend on the choice of the speeding constant. 
\end{abstract}


\section{A total version of Solovay reducibility}

We first review the usual definition of Solovay reducibility as a reducibility via a partially computable rational-valued function. For any unexplained notions, see Downey and Hirschfeldt~\cite{downey-hirschfeldt-2010}. 
\begin{definition}[Solovay~\cite{Solovay-1975}, 1975] 
Let~$\alpha$ and~$\beta$ be reals and let~$c> 0$ be a rational number. Then~$\alpha$  is \defhigh{Solovay reducible} to~$\beta$ \defhigh{with respect to a constant~$c$}, written $\alpha \leq_{S,c} \beta$, if there is a partial computable function $\varphi \colon  \mathbb{Q} \rightarrow \mathbb{Q}$ such that for all $q<\beta$ it holds that $\varphi(q)\downarrow<\alpha$ and $\alpha-\varphi(q)<c(\beta-q)$. The real~$\alpha$  is \defhigh{Solovay reducible} to $\beta$,  written $\alpha \leq_{S} \beta$, if $\alpha$  is Solovay reducible to~$\beta$ with respect to some~$c$.
\end{definition} 
In case~$\alpha \leq_{S} \beta$, we will also say that~$\alpha$ is \defhigh{$S$-reducible} to~$\beta$, and similar notation will be used for other reducibilities introduced in what follows.

\begin{definition}[\defhigh{total Solovay reducibility, $\leq^{tot}_{S,c}$}] 
A real $\alpha$  is \defhigh{total Solovay reducible} to a real $\beta$ \defhigh{with respect to a constant} $c$, written $\alpha \leq^{tot}_{S,c} \beta$, if there is a computable function $f \colon \mathbb{Q} \rightarrow \mathbb{Q}$ such that for all~$q<\beta$ it holds that $f(q)<\alpha$ and $\alpha-f(q)<c(\beta-q)$. The real $\alpha$  is \defhigh{total Solovay reducible} to~$\beta$,  written $\alpha \leq_S^{tot} \beta$, if~$\alpha$  is total Solovay reducible to~$\beta$ with respect to some~$c$.
\end{definition} 

The total Solovay reducibility obviously implies the normal one, thus, the Martin-Löf random left-c.e.\ reals are closed upwards relative to the total Solovay reducibility. One can proof that, in general case, the total Solovay reducibility implies the $\leq_S^{1a}$-reducibility defined by Zheng and Rettinger in~\cite{Rettinger-Zheng-2004}.

\section{The structural properties of the total Solovay degrees on the left-c.e.\ reals}
In this section, we argue that total Solovay reducibility is in $\Sigma^0_3$ but is not a standard reducibility in the sense of Downey and Hirschfeldt~\cite{downey-hirschfeldt-2010} because neither is addition a join operator nor is there a least degree. 
\begin{proposition}
Total Solovay reducibility is in $\Sigma^{0}_{3}$.
\end{proposition}
\begin{proof}
Let $\alpha^0,\alpha ^1,...$ be an effective enumeration of left-c.e.\ reals, where we can assume that, for given~$n$, on can compute a recursive index for a nondecreasing 
approximation~$a_0^n, a_1^n, \ldots$ to~$\alpha^n$ from below. Then we have
\begin{align*}
\alpha^a \leq_S^{tot}\alpha^b \iff& \exists \langle e,c\rangle \forall \langle q,s\rangle \exists \langle r,t \rangle \colon \\
&\bigg(\varphi_e(q)[t]\downarrow 
\wedge(q<b_s\implies\big(a_r-\varphi_e(q)>0\wedge a_r-\varphi_e(q)< c(b_s-q)\big)\bigg).
\end{align*}
\end{proof}

\begin{proposition}
Let~$\alpha$ be a left-c.e.\ real, and let~$r>0$ be a rational number. Then it holds that~$ r \alpha\equiv_S^{tot} \alpha$.
\end{proposition}

\begin{proof}
It holds that~$r \alpha\leq_S^{tot} \alpha$ via the identify function and constant~$r$, and similarly for a reduction in the reverse direction with constant~$1/r$.
\end{proof}

Next we review the notion of a hyperimmune set.
\begin{definition}
Let~$A$ be an infinite set. By~$p_A$, we denote the principal function of~$A$, i.e., the members of~$A$ are~$p_A(0)< p_A(1) < \cdots$. Let~$k_A(n)$ be the least member of~$A\setminus \{0, \ldots, n-1\}$. 
\end{definition}
Recall that a set~$A$ is~\defhigh{hyperimmune} if~$p_A$ is not majorized by a computable function, i.e., for no computable function~$g$ we have~$p_A(n)\le g(n)$ for all~$n$. 
\begin{lemma}\label{lemma:principal-function}
For any set~$A$, the following assertions are equivalent:
\begin{itemize}
\item[(i)] $p_A$ is not majorized by any computable function,
\item[(ii)] $k_A$ is not majorized by any computable function.
\end{itemize}
\end{lemma}
\begin{proof}
If the computable function~$g(n)$ majorizes $k_A(n)$, where we can assume that~$h$ in nondecreasing, then the desired computable function that majorizes $p_A(n)$ is given by
\[
n \mapsto \underbrace{g(g(...(g(0))...))}_{\text{$n$-fold application of~$g$}}.
\] 
Conversely, in case the computable function~$g(n)$ majorizes~$p_A(n)$, then the function $n \mapsto g(n+1)$ majorizes~$k_A(n)$.
\end{proof}
\begin{proposition}
There exists no least degree in the total Solovay degrees.
\end{proposition}
Every least set with respect to total Solovay reducibility is also a least set with respect to Solovay reducibility. Since the sets of the latter type are exactly the computable sets, the proposition is immediate from the following lemma. 

\begin{lemma}\label{lemma:computable-hyperimmune}
Let~$\alpha=0.A(0) \ldots$ and~$\beta=0.B(0) \ldots$ be reals where the set~$A$ is computable and infinite. Then~$\alpha$ is total Solovay reducible to~$\beta$ if and only if the set~$B$ is not hyperimmune. 
\end{lemma}
\begin{proof} 
First assume that~$B$ is not hyperimmune. For a dyadic rational~$q$ that can be written as~$q=0. \sigma$ where the last letter of~$\sigma$ is equal to~$1$, define~$|q|=|\sigma|$. Then, for any such~$q$ and~$\sigma$ where~$q < \beta$, we have
\[
2^{-k_B(|q|)}\le  \beta-0.\sigma =\beta-q.
\]
By Lemma~\ref{lemma:principal-function}, we can fix a computable function~$g$ that majorizes~$k_B$. We obtain a computable function $f$ witnessing $\alpha\leq_S^{tot}\beta$ by choosing~$f(q)< \alpha$ such that we have
\[\alpha - f(q)<  2^{-g(|q|)}.\]

Next assume that ~$\alpha$ is total Solovay reducible to~$\beta$ via some function~$f$ and constant~$c$. Then, for every~$n$ and for~~$q_n=0.B(0) \ldots B(n-1)$, we have~$q < \beta$, thus, for some appropriate constant~$d$, one holds that
\[
g(n):=\min_{|\sigma_n|=n}\{{\alpha-f(0.\sigma_n):\alpha-f(0.\sigma_n)>0}\}
\le \alpha-f(q) <  c (\beta -q) \le  2^{-k_B(n)+d}.
\]
Consequently, the function~$n \mapsto  d +\lceil \log g(n)\rceil$ is a computable upper bound for~$k_B$, hence~$B$ is not hyperimmune.
\end{proof}

Indeed, the total Solovay degree structure satisfies the following stronger property, which we state here without proof.
\begin{proposition}
There exists a countably infinite antichain of mutually $\leq_S^{tot}$-incomparable left-c.e.\ reals such that each of them is incomparable with every computable real.
\end{proposition}

Before proving that addition is not a join operator and not even an upper bound operator for total Solovay reducibility, we introduce the uniform version of the Schnorr reducibility introduced by Downey and Griffiths~\cite{downey-griffiths-2004} in 2004.
The possibility to define the uniform Schnorr reducibility has been noticed by Downey, Hirschfeldt~\cite[footnote to Definition~9.17.1]{downey-hirschfeldt-2010} in 2010.
\begin{definition}[Downey and Griffiths~\cite{downey-griffiths-2004}; Downey and Hirschfeldt~\cite{downey-hirschfeldt-2010}] 
A real $\alpha$  is \defhigh{uniformly Schnorr reducible}, or \defhigh{$uSch$-reducible}, to a real $\beta$ \defhigh{with respect to a constant} $c$, written $\alpha \leq_{uSch,c} \beta$, if there is a computable functional $\varphi$ that, given a description of a computable measure machine (or, shortly, cmm) $B$, returns a description of another cmm $\varphi(B)$, so that
\[K_{\varphi(B)}(\alpha\upharpoonright n)\leq K_B(\beta\upharpoonright n) + c\quad\text{for all }n.\]
The real~$\alpha$  is \defhigh{uniform Schnorr reducible} to $\beta$,  written $\alpha \leq_{uSch} \beta$, if $\alpha$  is uniform Schnorr reducible to~$\beta$ with respect to some~$c$.
\end{definition} 

Obviously, the uniform Schnorr reducibility implies the Schnorr reducibility, with respect to which the Schnorr random reals are closed upwards.

\begin{proposition}
For all left-c.e.\ $\alpha$ and $\beta$, $\alpha \leq_S^{tot} \beta$ implies $\alpha \leq_{uSch} \beta$.
\end{proposition}

\begin{corollary}
The Schnorr random left-c.e.\ reals are closed upwards under total Solovay reducibility.
\end{corollary}

\begin{proof}
Let $f$ be a total computable function, such that $\alpha \leq^{tot}_{S,c} \beta$ via $f$. Given a cmm machine $B$ computing $\beta$, we construct a cmm machine $A$ computing $\alpha$ in the following uniform way:

Input: $(x\in\mathbb{Q},w\in\{0,1\}^{\lceil log(c)+1 \rceil})$.
\begin{itemize}
\item Compute $\sigma:=B(x)$ (the computation halts iff $x\in dom(B)$).
\item Compute $\tau$ so that $0.\tau:=(f(0.\sigma)\upharpoonright n)$.

If $0.\sigma<\beta$, then on holds 
\[\alpha-f(0.\sigma)<c(\beta - 0.\sigma).\]
In particular, if $0.\sigma = \beta\upharpoonright n$, then $\beta-0.\sigma<2^{-n}$, so $\alpha-f(0.\sigma)<c2^{-n}=2^{\lceil log(c) \rceil -n}$.

Thus, \[\alpha\upharpoonright n-0.\tau<\alpha-f(0.\sigma)+f(0.\sigma)-(f(0.\sigma)\upharpoonright n)<c2^{-n}+2^{-n}=2^{\lceil log(c+1) \rceil -n}.\]

\item return $y\in\{0,1\}^n$, so that $0.y = 0.\tau+2^{-n} \cdot 0.w$
\end{itemize}

The constructed machine $A$ has the following properties:
\begin{itemize}
\item prefix-freeness since $B$ is prefix-free),
\item computable measure of the domain since the property
\[B(x)\downarrow \implies A\big((x,w)\big)\downarrow\forall w\in\{0,1\}^{\lceil log(c+1) \rceil}\]
implies, that $\mu\big(dom(A)\big)=\mu\big(dom(B)\big)$,
\item the inequality $K_A(\alpha\upharpoonright n)\leq K_B(\beta\upharpoonright n) + log(c+1)+O(1)$ for all~$n$ since there always exists a word $w\in\{0,1\}^{\lceil log(c+1) \rceil}$ such that \[\alpha\upharpoonright n-0.\tau = 2^{-n} \cdot 0.w,\]
hence, for that $w$, on holds $A(x,w) = y$, such that \[0.y=0.\tau+2^{-n} \cdot 0,w=\alpha\upharpoonright n,\]
that implies $K_A(\alpha\upharpoonright n)\leq|x|+|w|$, where $x$ may be the shortest code of $\beta\upharpoonright n$.
\end{itemize}
\end{proof}

\begin{remark}
    The uniform Schnorr reducibility is, due to the similar argumentation, also implied by the weaken version of the total Solovay reducibility, whose requirement for $f$ differs from the original one in the additional term:
    \[\alpha-f(q)<c(\beta-q)+2^{-|q|}.\]
    The motivation of this weakening is that now its lattice on the field of left-c.e.\ reals has a minimal degree containing all the computable reals.
\end{remark}

\begin{proposition}
There is a pair of left-c.e.reals $\alpha, \beta$ where $\alpha\nleq_S^{tot} \alpha+\beta$.
\end{proposition}

\begin{proof}
Miyabe, Nies and Stephan~\cite[Paragraph~3]{miyabe-etal-2018} demonstrated that there exists a pair of left-c.e.\ reals $\alpha$ and~$\beta$ such that $\alpha\nleq_{Sch} \alpha+\beta$. 
Thus we also have~$\alpha\nleq_S^{tot} \alpha+\beta$ because total Solovay reducibility implies Schnorr reducibility.
\end{proof}

\section{Speedability as a self-reducibility}
On the let of left-c.e.\ reals, the notion of speedability has been introduced by Merkle and Titov~\cite{Merkle-Titov-2020} in 2020 as a self-reducibility via respect to the constant $1$
\begin{definition}[Merkle and Titov~\cite{Merkle-Titov-2020}]\label{def:speedable}
    A function~$f\colon \naturals \rightarrow \naturals$ is a \defhigh{speed-up function} if it is nondecreasing and~$n \le f(n)$ holds for all~$n$. A left-c.e.\ number $\alpha$ is \defhigh{$\rho$-speedable with respect to its given left approximation} ~$a_0, a_1, \ldots \nearrow \alpha$ for some real number~$\rho\in [0,1)$ if there is a computable speed-up function~$f$ such that we have
\begin{equation}\label{eq:definition-speedable}
    \liminf\limits_{n\to\infty}\frac{\alpha-a_{f(n)}}{\alpha-a_n} \le \rho,
\end{equation}
    and \defhigh{speedable} if it is $\rho$-speedable with respect to some its left-c.e.\ approximation for some~$\rho \in (0,1)$.
    Otherwise we call $\alpha$ nonspeedable.
\end{definition}
Whether a real is speedable depends neither on the left-c.e.\ approximation nor on the constant $\rho$ one considers.

\begin{theorem}[Merkle and Titov~\cite{Merkle-Titov-2020}]\label{independance}
    Every speedable left-c.e.\ real number is $\rho$-speedable for any~${\rho>0}$ with respect to any of its left approximations.
\end{theorem}

The following theorem is immediate from the main result of Barmpalias and Lewis-Pye~\cite{barmpalias-lewispye-2017}.
\begin{theorem}[Barmpalias and Lewis-Pye~\cite{barmpalias-lewispye-2017}]
\label{theorem:ml-random-nonspeedable}
    Martin-Löf random left-c.e.\ real numbers are never speedable.
\end{theorem}

By the following proposition, the notion of speedability can be equivalently characterized as a Solovay reduction of a real number to itself via a special partial computable functions on the rational numbers.
By applying the same characterization to computable functions, in what follows we obtain a variant of speedability, similar to the introduction of total Solovay reducibility.

\begin{proposition}[Functional characterization of speedability]
\label{equivalent-definition}
Let~$\alpha$ be a left-c.e.\ real and let~$\rho$ be a real number such that~$0 < \rho <1$. Then $\alpha$ is speedable if and only if there is a partial computable function~$g \colon \mathbb{Q}\rightarrow \mathbb{Q}$ that is defined and nondecreasing on the interval~$(-\infty,\alpha)$, maps this interval to itself and satisfies
\begin{equation}\label{eq:equivalent-definition-liminf}
\liminf\limits_{q\nearrow\alpha}\frac{\alpha-g(q)}{\alpha-q}\leq\rho . 
\end{equation}
\end{proposition}

\begin{remark}
    Note that the characteristic of speedability given in Proposition~\ref{equivalent-definition} does not require the existence of a left-c.e.\ approximation of the real~$\alpha$, thus, can be considered as an extension of the speedability notion on all (i.e.\ not only left-c.e.) reals.
\end{remark}

\begin{proof}
Fix some left approximation~$a_0, a_1, \ldots$ of~$\alpha$. First, assume that~$\alpha$ is speedable. By Theorem~\ref{independance} there is then a computable speed-up function~$f$ that witnesses that~$\alpha$ is $\rho$-speedable with respect to its left approximation~$a_0, a_1, \ldots$ . Let~$g$ be the partial computable function on the set of rational numbers that maps every~$q < \alpha$ to the least index~$i$ such that~$q \le a_i$ and is undefined for all other~$q$. Here we assume that rational numbers are represented in a form such that equality is a computable predicate. Then the partial function~$g$ defined by
\[
g(q)= a_{f(n(q))},
\]
by choice of~$n$ and~$f$, is partial computable, is defined and nondecreasing on the interval~$(-\infty,\alpha)$, and maps this interval to itself. Furthermore, the sequence~$a_0, a_1, \ldots$ witnesses that~\eqref{eq:equivalent-definition-liminf} holds because we have~$g(a_i)= a_{f(i)}$.

Next assume that there is a function~$g$ as stated in the proposition. Then there is a not necessarily computable left approximation~$q_0, q_1, \ldots$ of~$\alpha$ such that we have
\[
\liminf\limits_{j \rightarrow \infty}\frac{\alpha-g(q_j)}{\alpha-q_j}\leq\rho .  
\]
Let~$f$ be the computable speed up function that maps~$i$ to the least index~$n>i$ such that~$g(a_{i+1}) < a_n$. Then, for all~$q$ and~$i$ such that~$q$ is an element of the half-open interval~$[a_{i}, a_{i+1})$, we have
\[
\frac{\alpha-a_{f(i)}}{\alpha-a_{i}} \le
\frac{\alpha-g(a_{i+1})}{\alpha-q} \le
\frac{\alpha-g(q)}{\alpha-q}.
\]
In particular, this chain of inequalities holds true with~$q$ replaced by any of the~$q_j$, which, by choice of the~$q_j$, implies that~$\alpha$ is $\rho$-speedable via its left approximation~$a_0, a_1, \ldots$ and the speed-up function~$f$.
\end{proof}
From Proposition~\ref{equivalent-definition} it is immediate that the equivalent characterization of speedability stated here does not depend on the choice of~$\rho$ in the interval~$(0,1)$. In particular, the characterization holds for some~$\rho$ in this interval if and only if it holds for all~$\rho$ in this interval.

In a same way as the totalizing of translation function for the Solovay reducibility, we can totalize the concept of speedability by requiring the function $g$ from the latter definition to be total.
\begin{definition}
\label{total speedable}
Let~$\rho$ be a real number such that~$0 < \rho <1$. A left-c.e.\ real $\alpha$ is called \defhigh{total $\rho$-speedable} if there exists a nondecreasing computable function $g \colon \mathbb{Q}\mapsto\mathbb{Q}$ that maps every~$q$ in the interval~$(-\infty,\alpha)$ to a value~$g(q) > q$ in this interval and satisfies
\begin{equation}\label{eq:equivalent-definition-liminf-total}
\liminf\limits_{q\nearrow\alpha}\frac{\alpha-g(q)}{\alpha-q}\leq\rho . 
\end{equation}
Such a function $g$ is called \defhigh{total speed-up function}.
\end{definition}

By the following proposition, the total version of speedability does again not depend on the choice of the constant. The proof is omitted due to space considerations.

\begin{proposition}
    Whether a left-c.e.\ real is total $\rho$-speedable does not depend on the choice of $\rho\in(0,1)$.
\end{proposition}
Barmpalias and Lewis-Pye~\cite{barmpalias-lewispye-2017} have shown that speedability implies  Martin-Löf nonrandomness. We currently research the characteristics of the total speed-up function via which the total speedability will imply Schnorr nonrandomness.



\begin{thebibliography}{1}

\bibitem{barmpalias-lewispye-2017}
George Barmpalias and Andrew Lewis-Pye.
Differences of halting probabilities.
\emph{Journal of Computer and System Sciences} 89:349--360 (2017).

\bibitem{downey-griffiths-2004}
Rodney Downey and Evan Griffiths.
Schnorr randomness.
\emph{Journal of Symbolic Logic} 69(2):533--554 (2004).

\bibitem{downey-hirschfeldt-2010}
Rodney Downey and Denis Hirschfeldt.
\emph{Algorithmic Randomness and Complexity},
Springer, Berlin (2010).

\bibitem{Merkle-Titov-2020}
Wolfgang Merkle and Ivan Titov.
Speedable left-c.e.\ numbers.
\emph{CSR 2020: Computer Science –- Theory and Applications} pp 303--313 (2020).

\bibitem{miller-2016}
Joseph Miller.
On work of Barmpalias and Lewis-Pye: A derivation on the d.c.e.\ reals.
\emph{Lecture Notes in Computer Science} 10010:644–-659 (2016).

\bibitem{miyabe-etal-2018}
Kenshi Miyabe, André Nies, and Frank Stephan.
Randomness and Solovay degrees.
\emph{Journal of Logic and Analysis} 10(3):1--13 (2018).

%
\bibitem{Solovay-1975}
Robert Solovay.
Draft of paper (or series of papers) on Chaitin’s work. 
\emph{Unpublished notes}, 215 pages (1975)

\bibitem{Rettinger-Zheng-2004}
Xizhong Zheng and Robert Rettinger.
On the Extensions of Solovay-Reducibility.
\emph{COCOON 2004: Computing and Combinatorics} pp 360--369 (2004).

\end{thebibliography}
\end{document}